\documentclass[12pt]{amsart}
\usepackage{amsmath,latexsym,amscd,amsbsy,amssymb,amsfonts,amsthm,fleqn,leqno,
euscript, graphicx, texdraw, pb-diagram,tikz}
\usepackage[matrix,arrow,curve]{xy}

\numberwithin{equation}{section}
\newtheorem{thm}{Theorem}[section]
\newtheorem{pro}[thm]{Proposition}
\newtheorem{lem}[thm]{Lemma}
\newtheorem{cor}[thm]{Corollary}

\theoremstyle{definition}

\theoremstyle{remark}

\numberwithin{equation}{section}

\begin{document}

\title[Homological selections and fixed-point theorems]
{Homological selections and fixed-point theorems}

\author{V. Valov}
\address{Department of Computer Science and Mathematics,
Nipissing University, 100 College Drive, P.O. Box 5002, North Bay,
ON, P1B 8L7, Canada} \email{veskov@nipissingu.ca}

\thanks{The author was partially supported by NSERC
Grant 261914-13.}

 \keywords{fixed points, homological $UV^n$ sets, homological selections}

\subjclass[2010]{Primary 54H25, 55M20; Secondary 55M10, 55M15}
\begin{abstract}
A homological selection theorem for $C$-spaces, as well as a finite-dimensional homological selection theorem is established.
We apply the finite-dimensional homological selection theorem to obtain fixed-point theorems for usco homologically $UV^n$ set-valued maps.
\end{abstract}
\maketitle\markboth{}{Homological selections}





\section{Introduction}

 Banakh and Cauty \cite[Theorem 8]{bc} provided a selection theorem for $C$-spaces, which is a homological version of the Uspenskij's selection theorem \cite[Theorem 1.3]{us}. The aim of this paper was to establish a finite-dimensional form of Banach-Cauty theorem, which is the main tool in proving homological analogues of  fixed-point theorems for  usco maps established in \cite{bo}, \cite{ggk} and \cite{vg2}.

All spaces are assumed to be completely regular. Singular homology $H_n(X;G)$, reduced in dimension 0, with a coefficient group $G$ is considered everywhere below. By default, if not explicitly stated otherwise,  $G$ is a ring with unit $\rm e$.
Following the notations from \cite{bc}, for any space $X$ let $S_k(X;G)$ be the group of all singular chains with coefficients from $G$ consisting of singular $k$-simplexes and $S(X;G)$ denote the singular complex of $X$, so $S(X;G)$ is the direct sum $\bigoplus_{k=0}^{\infty}S_k(X;G)$. The groups $S_k(X;G)$ are linked via the boundary homomorphisms $\partial:S_k(X;G)\to S_{k-1}(X;G)$.

If $\sigma:\triangle^k\to X$ is a singular $k$-simplex ($\triangle^k$ is the standard $k$-simplex), we denote by $||\sigma||$ the {\em carrier} $\sigma(\triangle^k)$ of $\sigma$. Similarly, we put $||c||=\bigcup_i||\sigma_i||$ for any chain $c\in S_k(X;G)$, where $c=\sum_ig_i\sigma_i$ is the irreducible representation of $c$.

For an open cover $\mathcal U$ of $X$ let $S(X,\mathcal U;G)$ stand for the subgroup of $S(X;G)$ generated by singular simplexes $\sigma$ with $||\sigma||\subset U$ for some $U\in\mathcal U$. A homomorphism $\varphi:S(X,\mathcal U;G)\to S(Y;G)$ is said to be a {\em chain morphism} if
$\varphi(S_k(X,\mathcal U;G))\subset S_k(Y;G)$ for all $k\geq 0$ and $\partial\circ\varphi=\varphi\circ\partial$. A point $x\in X$ is called a {\em fixed point} for a chain morphism $\varphi:S(X,\mathcal U;G)\to S(X;G)$ if for any neighborhood $V$ of $x$ in $X$ there exists a chain $c\in S(V;G)\cap S(X,\mathcal U;G)$ such that $||\varphi(c)||\cap V\neq\varnothing$.

Let $A\subset B$ be two subsets of a space $X$. We write $A\stackrel{H_m}{\hookrightarrow} B$
if the embedding $A\hookrightarrow B$ induces a trivial homomorphism $H_m(A;G)\to H_m(B;G)$.

A set-valued map $\Phi:X\to 2^Y$ is called {\em strongly lower semi-continuous} (br., strongly lsc) if for each compact subset $K\subset Y$ the set $\{x\in X: K\subset\Phi(x)\}$ is open in $X$. For example, every open-graph set-valued map $\Phi:X\to 2^Y$ is strongly lsc, see \cite[Proposition 3.2]{vg1}.

Here is our first homological selection theorem.

\begin{thm}
Let $X$ be a paracompact $C$-space, $Y$ be an arbitrary space and $\Phi_k:X\to 2^Y$, $m=0,1,..,n$, be a finite sequence of strongly lsc maps satisfying the following conditions, where $G$ is a fixed ring with unit:
\begin{itemize}
\item[(i)] $\Phi_m(x)\stackrel{H_m}{\hookrightarrow}\Phi_{m+1}(x)$ for every $m=0,..,n-1$ and every $x\in X$;
\item[(ii)] $H_{m}(\Phi_n(x);G)=0$ for all $m\geq n$ and all $x\in X$.
\end{itemize}
Then there exists an open cover $\mathcal U$ of $X$ and a chain morphism $\varphi: S(X,\mathcal U;G)\to S(Y;G)$ such that $\varphi(S(U;G))\subset S(\Phi_n(x);G)$ for every $U\in\mathcal U$ and every $x\in U$.
\end{thm}
Let us mention that the Banakh-Cauty result \cite[Theorem 8]{bc} is a particular case of Theorem 1.1 with $\Phi_m=\Phi_n$ for all $m$.
There is also a finite-dimensional analogue of the above theorem.
\begin{thm}
Let $X$, $Y$ and $G$ be as in Theorem $1.1$. The same conclusion holds if $\dim X\leq n$ and the sequence of strongly lsc maps $\Phi_m:X\to 2^Y$ satisfies only condition $(i)$.
\end{thm}

Theorem 1.2 and \cite[Theorem 7]{bc} imply the following fixed point theorem for usco (upper semi-continuous and compact-valued) maps:
\begin{thm}
Let  $X$ be a paracompact space with $\dim X\leq n$, $Y$ a compact metric $AR$, $G$ a field and $\Phi:X\to 2^Y$ be a homologically $UV^{n-1}(G)$ usco map. Then for every continuous map $g:Y\to X$ there exists a point $y_0\in Y$ with $y_0\in\Phi(g(y_0))$.
\end{thm}
The particular case of Theorem 1.3 with $X=Y$ and $g=id_X$ is also interesting.
\begin{cor}
Let  $X$ be a compact metric $AR$ with $\dim X\leq n$, $G$ a field and $\Phi:X\to 2^X$ be a homologically $UV^{n-1}(G)$ usco map. Then there exists a point $x_0\in X$ with $x_0\in\Phi(x_0)$.
\end{cor}
Recall that a closed subset $A$ of a metric space $X$ is called $UV^n$ in $X$ if every neighborhood $U$ of $A$ in $X$ contains another neighborhood $V$ such that the inclusion $V\hookrightarrow U$ generates trivial homomorphisms $\pi_k(V)\to\pi_k(U)$ between the homotopy groups for all $k=0,..n$.
If considering the homology groups $H_k(.;G)$ instead of the homotopy  groups $\pi_k(V)$ and $\pi_k(U)$ (i.e. requiring
$V\stackrel{H_m}{\hookrightarrow}U$ for all $m=0,1,..,n$), then $A$ is said to be {\em homologically $UV^n(G)$ in $X$}. It follows from the universal formula coefficients that every $UV^n$-subset of $X$ is homologically $UV^n(G)$ in $X$ for all groups $G$. Moreover, following the proof of Proposition 7.1.3 from \cite{sa}, one can show that $A$ is homologically $UV^n(G)$ in a given metric $ANR$-space $X$ if and only if it is homologically $UV^n(G)$ in any metric $ANR$-space that contains $A$ as a closed set. We say that $A$ is {\em homologically $UV^\omega(G)$ in $X$} if $A$ is homologically $UV^n(G)$ for all $n\geq 0$. Because of the last two notions, it is convenient to say that $A$ is {\em homotopically $UV^n$ in $X$} instead of $A$ being $UV^n$ in $X$. We also say that a set-valued map $\Phi:X\to 2^Y$ is {\em homologically $UV^n(G)$} if all values $\Phi(x)$ are homologically $UV^n(G)$-subsets of $Y$.
\begin{thm}
Let  $X$ be a compact metric $AR$-space, $G$ a field and $\Phi:X\to 2^X$ be a homologically $UV^{\omega}(G)$ usco map. Then there exists a point $x_0\in Y$ with $x_0\in\Phi(x_0)$.
\end{thm}

 Theorem 1.3 was established by Gutev \cite{vg2} for usco maps with homotopically $UV^n$ values. A homotopical version of Theorem 1.5 is also known, see Corollaries 3.6 and 5.14 from \cite{ggk}, or Theorem 1.3 from \cite{vg2}. One can also show that if $X$ is a compact metric $AR$ and $\Phi:X\to 2^X$ is a homologically $UV^{\omega}(G)$ usco map, then each value $\Phi(x)$ has trivial \v{C}ech homology groups with coefficients in $G$. So, in the particular case when $G$ is the group $\mathbb Q$ of the rationals, Theorem 1.5 follows from the more general \cite[Theorem 7]{ca} treating the so-called algebraic $AR$'s. However, in the framework of usual $AR$'s Theorem 1.5 provides a very simple proof.

\section{Homological selection theorems}

In this section we prove Theorems 1.1 - 1.2. For any simplicial complex $K$ and an integer $m\geq 0$ let $K^{(m)}$ and $C_m(K;G)$ denote, respectively, the $m$-skeleton of $K$ and the group generated by the oriented $m$-simplexes of $K$ with coefficients in $G$.

We say that a chain morphism $\mu:C(K;G)\to S(A;G)$ (resp., $\mu:S(A;G)\to C(K;G)$), where $K$ is a simplicial complex and $A$ a topological space, is
{\em correct} provided $\mu(v)$ is a singular $0$-simplex in $S(A;G)$ for every vertex $v\in K^{(0)}$ (resp., $\mu(\sigma)$ is a vertex of $K$ for every singular $0$-simplex $\sigma\in S_0(A;G)$).

\begin{lem}
Suppose $\{A\}_{k=0}^{m+1}$ is a sequence of subsets of $Y$  with $A_k\stackrel{H_k}{\hookrightarrow}A_{k+1}$, $k=0,1,..,m$. Let
$L$ be a simplicial complex of dimension $m$ and $K$ be the cone of $L$. Then every correct chain morphism $\mu_m:C(L;G)\to S_m(A_m;G)$ such that $\mu_m(C(L^{(k)};G))\subset S_k(A_k;G)$ for all $k\leq m$
can be extended to a correct chain morphism
 $\mu_{m+1}:C(K;G)\to S_{m+1}(A_{m+1};G)$ satisfying the following conditions:
 \begin{itemize}
\item $\mu_{m+1}(C(K^{(k)};G))\subset S_k(A_k;G)$ for all $k=0,1,..,m+1$;
\item $\widetilde{\mu}_m\circ\partial_{m+1}=\partial_{m+1}\circ\mu_{m+1}$, where
 $\widetilde{\mu}_m=\mu_{m+1}|C(K^{(m)};G)$.
 \end{itemize}
\end{lem}

\begin{proof}
We first extend each morphism $\mu_k=\mu_m|(C(L^{(k)};G)$ to a morphism $\widetilde{\mu}_k:C(K^{(k)};G)\to S_k(A_k;G)$ such that $\widetilde{\mu}_k\circ\partial_{k+1}=\partial_{k+1}\circ\widetilde{\mu}_{k+1}$, $k=0,1,,m-1$. To this end, denote by $v_0$ the vertex of $K$ and consider the augmentation
$\epsilon:S_0(A_1;G)\to G$ defined by $\epsilon(\sigma)=\rm e$ for all singular $0$-simplexes $\sigma\in S_0(A_1;G)$. Define $\widetilde{\mu}_0(\{v_0\})$ to be a fixed singular simplex $\sigma_0\in S_0(A_0;G)$ and $\widetilde{\mu}_0(\{v\})=\mu_0(\{v\})$ for any $v\in L^{(0)}$. Then extend $\mu_0$ to a homomorphism
 $\widetilde{\mu}_0:C(K^{(0)};G)\to S_0(A_0;G)$ by linearity.

If $\sigma=(v_1,v_2)$ is an $1$-dimensional simplex in $K$, then $\widetilde{\mu}_0(\partial_1(\sigma))=\widetilde{\mu}_0(v_2)-\widetilde{\mu}_0(v_1)$. Hence, $\epsilon(\widetilde{\mu}_0(\partial_1(\sigma)))=0$.
 Since $A_0\stackrel{H_0}{\hookrightarrow}A_{1}$, there is a singular chain $\tau_\sigma\in S_{1}(A_1;G)$ such that $\widetilde{\mu}_0(\partial_1(\sigma))=\partial_1(\tau_\sigma)$. Letting $\widetilde{\mu}_{1}(\sigma)=\tau_\sigma$ if $\sigma\in K^{(1)}\setminus L^{(1)}$ and
$\widetilde{\mu}_{1}(\sigma)=\mu_{1}(\sigma)$ if $\sigma\in L^{(1)}$, we define the homomorphism $\widetilde{\mu}_{1}$ on every simplex of $K^{(1)}$. Then extend this homomorphism  to  $\widetilde{\mu}_1:C(K^{(1)};G)\to S_1(A_1;G)$ by linearity. 

Because $A_{k-1}\stackrel{H_{k-1}}{\hookrightarrow}A_{k}$, we can repeat the above construction to obtain the homomorphisms $\widetilde{\mu}_k$ for any $k\leq m$. Then $\widetilde{\mu}_m:C(K^{(m)};G)\to S_m(A_m;G)$. Since $A_{m}\stackrel{H_{m}}{\hookrightarrow}A_{m+1}$, we can use once more the above arguments to obtain the chain morphism  $\mu_{m+1}:C(K;G)\to S_{m+1}(A_{m+1};G)$ satisfying the required conditions.
\end{proof}

\begin{lem}
Let $L$ be a simplicial complex with trivial homology groups and $A\subset B$ be a pair of spaces. Then every correct chain morphism
$\nu: S(A;G)\to C(L;G)$ can be extended to a correct chain morphism $\widetilde{\nu}:S(B;G)\to C(L;G)$.
\end{lem}

\begin{proof}
We are going to define by induction for each $k\geq 0$ a homomorphism $\widetilde{\nu}_k:S_k(B;G)\to C_k(L;G)$ extending $\nu_k:S_k(A;G)\to C_k(L;G)$ such that $\widetilde{\nu}_k(\partial_{k+1}(c))=\partial_{k+1}(\widetilde{\nu}_{k+1}(c))$ for every singular chain $c\in S_{k+1}(B;G)$ and $k\geq 0$.
For every singular 0-simplex $\sigma\in S_0(B;G)$ we define $\widetilde{\nu}_0(\sigma)=\nu_0(\sigma)$ if $\sigma\in S_0(A;G)$ and $\widetilde{\nu}_0(\sigma)=v_0$ if   $\sigma\not\in S_0(A;G)$, where $v_0$ is a fixed vertex of $L$. Then extend this homomorphism over $S_0(B;G)$ by linearity. Because $\nu$ is correct, $\widetilde{\nu}_0(\sigma)$ is a vertex of $L$ for all singular $0$-simplexes $\sigma\in S(B;G)$.

To define $\widetilde{\nu}_1$ we
consider the augmentation $\epsilon:C_0(L;G)\to G$ defined by $\epsilon(v)=\rm e$ for all vertexes of $L$, see \cite{sp}. Thus,
$\epsilon(\widetilde{\nu}_0(\partial_1(\sigma)))=0$ for every singular 1-simplex $\sigma\in S_1(B;G)$.
Because $H_0(L;G)=0$,
$\partial_1(C_1(L;G))=\epsilon^{-1}(0)$. Therefore, for every singular simplex $\sigma\in S_1(B;G)\setminus S_1(A;G)$ there exists a chain $c_\sigma\in C_1(L;G)$ such that $\partial_1(c_\sigma)=\widetilde{\nu}_0(\partial_1(\sigma))$. We define $\widetilde{\nu}_1(\sigma)=\nu_1(\sigma)$ if $\sigma\in S_1(A;G)$ and
$\widetilde{\nu}_1(\sigma)=c_\sigma$ if $\sigma\in S_1(B;G)\setminus S_1(A;G)$, and extend $\widetilde{\nu}_1$ over $S_1(B;G)$ by linearity.

Suppose the homomorphism $\widetilde{\nu}_k:S_k(B;G)\to C_k(L;G)$ was already constructed.
Then, using that the kernel of the boundary homomorphism $\partial_k:C_k(L;G)\to C_{k-1}(L;G)$ coincides with the image $\partial_{k+1}(C_{k+1}(L;G))$, we can define $\widetilde{\nu}_{k+1}$ extending $\widetilde{\nu}_k$ and satisfying the equality $\widetilde{\nu}_k\circ\partial_{k+1}=\partial_{k+1}\circ\widetilde{\nu}_{k+1}$.
\end{proof}

\smallskip
\textit{Proof of Theorem $1.1$.} We modify the proof of \cite[Theorem 8]{bc}. By induction we are going to construct two sequences of locally finite open covers of $X$,
$\mathcal V_m=\{V_\alpha:\alpha\in\Gamma_m\}$ and $\mathcal W_m=\{W_\alpha:\alpha\in\Gamma_m\}$, $m\geq 0$, an increasing sequence $K_0\subset K_1\subset...$ of simplicial complexes and correct chain morphisms $\mu_m:C(K_m;G)\to S_m(Y;G)$ such that
\begin{itemize}
\item[(1)] $\overline{W}_\alpha\subset V_\alpha$ for all $\alpha\in\Gamma_m$, $m\geq 0$;
\item[(2)] $\dim K_m=m$;
\item[(3)] $\mu_{m+1}|C(K_{m};G)=\mu_{m}$ and $\partial_{m+1}\circ\mu_{m+1}=\widetilde{\mu}_{m}\circ\partial_{m+1}$, where
$\widetilde{\mu}_{m}=\mu_{m+1}|C(K_m^{(m)};G)$.
\end{itemize}

Moreover, for every $m$ and $\alpha\in\Gamma_m$ we shall assign a finite sub-complex $L_\alpha$ of $K_m$ and a set
$\Omega_\alpha=\bigcup_{\sigma\in L_\alpha}||\mu_m(\sigma)||$ satisfying the following conditions:
\begin{itemize}
\item[(4)] $\dim L_\alpha=m$ and $L_\alpha$ is a cone whose base is a sub-complex $M_\alpha\subset K_{m-1}$ and having a vertex $\alpha$;
\item[(5)] If $m\leq n$ and $\alpha\in\Gamma_m$, then $\Omega_\alpha\subset\Phi_m(x)$ and $\Omega_\alpha^{(k)}\subset\Phi_k(x)$ for all $k\leq m-1$ and all $x\in V_\alpha$, where $\Omega_\alpha^{(k)}=\bigcup_{\sigma\in L_\alpha}||\mu_{m}(\sigma^{(k)})||$;
\item[(6)] If $m>n$ and $\alpha\in\Gamma_m$, then $\Omega_\alpha\subset\Phi_n(x)$ for all $x\in V_\alpha$.
\end{itemize}

To start our construction, for every $x\in X$ we fix a point $y_x\in\Phi_0(x)$ and consider the set $O_x=\{x'\in X:y_x\in\Phi_0(x')\}$. Since $\Phi_0$ is strongly lsc, $O_x$ is open in $X$. Let $\mathcal V_0=\{V_\alpha:\alpha\in\Gamma_0\}$ be a locally finite open cover of $X$ refining the cover
$\{O_x:x\in X\}$, and choose $\mathcal W_0=\{W_\alpha:\alpha\in\Gamma_0\}$ to be a locally finite open cover of $X$ with $\overline{W}_\alpha\subset V_\alpha$ for all $\alpha\in\Gamma_0$. Let the complex $K_0$  be the zero-dimensional complex whose set of vertices is $\Gamma_0$. For every $\alpha\in\Gamma_0$ we set $L_\alpha=\{\alpha\}$ and choose $x_\alpha\in X$ such that $V_\alpha\subset O_{x_\alpha}$. Define $\mu_0:C(K_0;G)\to S_0(Y;G)$ to be the homomorphism assigning to each generator corresponding to $\alpha$ the singular 0-simplex $y_{x_\alpha}$, and let $\Omega_\alpha=\{y_{x_\alpha}\}.$ Obviously, $\mu_0$ is correct.

Suppose for some $m<n-1$ and all $k\leq m$ we already performed the construction satisfying conditions $(1) - (5)$. Then for every $x\in X$ choose an open neighborhood $G_x$ of $x$ meeting only finitely many elements of the cover $\bigcup_{k\leq m}\mathcal V_k$ such that $G_x\subset V_\alpha$ for all
$\alpha\in\bigcup_{k=0}^{m}\Gamma_k$ with $G_x\cap\overline W_\alpha\neq\varnothing$. Let $J(x)=\{\alpha\in\bigcup_{k=0}^{m}\Gamma_k: G_x\subset V_\alpha\}$ and $D_x^{(k)}=\bigcup\{\Omega_\alpha^{(k)}:\alpha\in J(x)\}$, $k\leq m$. Since $J(x)$ is finite, all $D_x^{(k)}$ are compact subsets of $Y$ with $D_x^{(k)}\subset D_x^{(m)}=D_x$. Moreover, condition $(5)$ implies $D_x\subset\Phi_m(x)$. Consider the finite sub-complex $M_x=\bigcup\{L_\alpha:\alpha\in J(x)\}$ of $K_m$ and the cone $L_x$ with a vertex $v_x\not\in K_m$ and a base $M_x$. Then, according to the definition of $\Omega_\alpha$ and condition $(5)$, we have $\mu_m(C(M_x^{(k)};G))\subset S_k(D_x^{(k)};G)\subset S_k(\Phi_k(x);G)$, $k\leq m$. Therefore,  we can apply Lemma 2.1 to find a correct chain morphism $\mu_x:C(L_x;G)\to S_{m+1}(\Phi_{m+1}(x);G)$ extending $\mu_m|C(M_x;G)$ such that $\mu_x(C(L_x^{(k)};G))\to S_{k}(\Phi_{k}(x);G)$ and
$\partial_{m+1}\circ\mu_{x}=(\mu_{x}|C(L_x^{(m)};G))\circ\partial_x$, where $\partial_x:C(L_x;G)\to C(L_x^{(m)};G)$ is the boundary homomorphism.
Then $\Omega_x=\bigcup_{\sigma\in L_x}||\mu_x(\sigma)||$ is a compact subset of $\Phi_{m+1}(x)$ containing $D_x$. The strong lower semi-continuity of $\Phi_{m+1}$ yields that $O^m_x=\{x'\in G_x:\Omega_x\subset\Phi_{m+1}(x')\}$ is an open neighborhood of $x$. So, there exists a locally finite open cover $\mathcal V_{m+1}=\{V_\alpha:\alpha\in\Gamma_{m+1}\}$ of $X$ refining the cover $\{O^m_x: x\in X\}$, and take a locally finite open cover $\mathcal W_{m+1}=\{W_\alpha:\alpha\in\Gamma_{m+1}\}$ satisfying condition $(1)$. Now, for every $\alpha\in\Gamma_{m+1}$ choose $x_\alpha\in X$ with
$V_\alpha\subset O^m_{x_\alpha}$ and let $L_\alpha$ be the cone with base $M_{x_\alpha}$ and vertex $\alpha$. Define $K_{m+1}$ to be the union
$K_m\cup\bigcup_{\alpha\in\Gamma_{m+1}}L_\alpha$. Identifying the cones $L_\alpha$ and $L_{x_\alpha}$, we define the correct morphism
$\mu_{m+1}:C(K_{m+1};G)\to S_{m+1}(Y;G)$ by $\mu_{m+1}|C(K_m;G)=\mu_{m}$ and $\mu_{m+1}|C(L_\alpha;G)=\mu_{x_\alpha}$. Finally, let $\Omega_\alpha=\Omega_{x_\alpha}$. It is easily seen that conditions $(1) - (5)$ are satisfied. Moreover, the definition of $G_x$ and the inclusion  $O^m_{x_\alpha}\subset G_{x_\alpha}$ yield that
\begin{itemize}
\item[(7)] For every $\beta\in\bigcup_{k=0}^{m}\Gamma_k$ and $\alpha\in\Gamma_{m+1}$ with $W_\alpha\cap W_\beta\neq\varnothing$ we have $W_\alpha\subset V_\beta$, and thus $L_\beta\subset L_\alpha$.
\end{itemize}
In this way we can perform our construction for all $m\leq n$. If we substitute $\Phi_m=\Phi_n$ for all $m\geq n$, we have also $\Phi_m(x)\stackrel{H_m}{\hookrightarrow}\Phi_{m+1}(x)$ because $H_m(\Phi_n(x);G)=0$ for all $x\in X$. Therefore, following the above arguments, we can perform for all $m$ the steps from $m$ to $m+1$ satisfying conditions $(1) - (7)$.

Let $K=\bigcup_{m=0}^{\infty}K_m$. Then the morphisms $\mu_m$ define a correct chain morphism $\mu:C(K;G)\to S(Y;G)$.
Because $X$ is a $C$-space, there exists a sequence of disjoint open families $\mathcal U_m=\{U_\lambda:\lambda\in\Lambda_m\}$, $m\geq 0$, such that each $\mathcal U_m$ refines $\mathcal W_m$ and the family $\mathcal U=\bigcup_{m=0}^{\infty}\mathcal U_m$ covers $X$.

The final step of the proof is to construct of a chain morphism from
$S(X,\mathcal U;G)$ into $C(K;G)$. To this end, let $\Lambda=\bigcup_{k=0}^{\infty}\Lambda_k$ and $\Lambda(m)=\bigcup_{k=0}^{m}\Lambda_k$. Consider also the sub-complexes $S_m=\sum_{\lambda\in\Lambda(m)}S(U_\lambda;G)$ of $S(X;G)$, $m\geq 0$, whose union is $S(X,\mathcal U;G)$. For every $\lambda\in\Lambda_m$ select an $\alpha_\lambda\in\Gamma_m$ with $U_\lambda\subset W_{\alpha_\lambda}$. We are going to construct a correct chain morphism $\nu:S(X,\mathcal U;G)\to C(K;G)$ such that
\begin{itemize}
\item[(8)] $\nu(S(U_\lambda;G))\subset C(L_{\alpha_\lambda};G)$ for all $\lambda\in\Lambda$.
\end{itemize}
For any $\lambda\in\Lambda_0$ the complex $L_{\alpha_\lambda}$ is a single point. So, we can find a chain morphism $\nu_\lambda:S(U_\lambda;G)\to C(L_{\alpha_\lambda};G)$. Since the family $\mathcal U_0$ is disjoint, $S_0$ is the direct sum of all
$S(U_\lambda;G)$, $\lambda\in\Lambda_0$. Hence, the chain morphism $\nu_0:S_0\to C(K;G)$ with $\nu_0|S(U_\lambda;G)=\nu_\lambda$ for all $\lambda\in\Lambda_0$ is well defined and $\nu_0(S(U_\lambda;G))\subset C(L_{\alpha_\lambda};G)$.

Suppose that for some $m$ we have constructed correct chain morphisms $\nu_k:S_k\to C(K;G)$, $k\leq m$, such that $\nu_{k}$ extends $\nu_{k-1}$ and $\nu_k(S(U_\lambda;G))\subset C(L_{\alpha_\lambda};G)$ for all $\lambda\in\Lambda(k)$.
Because $\mathcal U_{m+1}$ is a disjoint family, so is the family $\{S(U_\lambda;G):\lambda\in\Lambda_{m+1}\}$. Therefore, to extend $\nu_m$ over $S_{m+1}$, it suffices for every $\lambda\in\Lambda_{m+1}$ to extend $\nu_m|(S(U_\lambda;G)\cap S_m)$ over $S(U_\lambda;G)$. To this end, observe that if $\lambda\in\Lambda_{m+1}$ and $\lambda'\in\Lambda(m)$ with $U_\lambda\cap U_{\lambda'}\neq\varnothing$, then $W_{\alpha_\lambda}\cap W_{\alpha_{\lambda'}}\neq\varnothing$. Thus, according to condition $(7)$, $L_{\alpha_{\lambda'}}\subset L_{\alpha_{\lambda}}$. Consequently, by $(8)$, $\nu_m(S(U_\lambda;G)\cap S_m)\subset C(L_{\alpha_\lambda};G)$ for any $\lambda\in\Lambda_{m+1}$. Since $L_{\alpha_\lambda}$ is contractible and $\nu_m$ is correct, we can apply Lemma 2.2 (with
$A=U_\lambda\cap\bigcup_{\lambda'\in\Lambda(m)}U_{\lambda'}$ and $B=U_{\lambda}$) to find a correct chain morphism $\nu_\lambda:S(U_\lambda;G)\to C(L_{\alpha_\lambda};G)$ extending $\nu_m|S(U_\lambda;G)$. This completes the induction, so the construction of the required chain morphism
$\nu:S(X,\mathcal U;G)\to C(K;G)$ is done.

Finally, let $\varphi:S(X,\mathcal U;G)\to S(Y;G)$ be the composition $\varphi=\mu\circ\nu$. Then, according to $(7)$ and the definitions of $\Omega_\alpha$,  for every $\lambda\in\Lambda$ we have
$$\varphi(S(U_\lambda;G))\subset\mu(C(L_{\alpha_\lambda};G))\subset S(\Omega_{\alpha_\lambda};G).$$
Since $U_\lambda\subset W_{\alpha_\lambda}$, conditions $(4)$ and $(5)$ yield that $\Omega_{\alpha_\lambda}\subset\Phi_n(x)$ for all
$x\in U_\lambda$. Therefore, $\varphi(S(U_\lambda;G))$ is contained in $S(\Phi_n(x);G)$ whenever $x\in U_\lambda$. $\Box$

\smallskip
\textit{Proof of Theorem $1.2$.} Since the sequence $\{\Phi_m\}_{m=0}^{n}$ satisfies condition $(i)$ from Theorem 1.1, we can perform  the construction from the proof of Theorem 1.1 for every $m=0,1,..,n$. So, we construct the locally finite covers $\mathcal V_m=\{V_\alpha:\alpha\in\Gamma_m\}$ and $\mathcal W_m=\{W_\alpha:\alpha\in\Gamma_m\}$ of $X$, the complexes $K_0\subset K_1\subset...\subset K_n$, the sets $\Omega_\alpha$ for any $\alpha\in\bigcup_{k=0}^{n}\Gamma_k$ and the correct chain morphisms $\mu_m:C(K_m;G)\to S_m(Y;G)$ satisfying conditions $(1)-(5)$ and the particular case of condition $(7)$ with $m\leq n-1$. Since the complex $K=\bigcup_{m=0}^n K_m$ is $n$-dimensional, $K^{(m)}=\varnothing$ for all $m>n$. So, we can suppose that $\mu_m=\mu_n$ for $m\geq n$. In this way we obtains
a chain morphism $\mu: C(K;G)\to S(Y;G)$. Because $\dim X\leq n$, according to
Corollary 5.3 from \cite{gv}, for every $m=0,1,..,n$ there exists a disjoint family $\mathcal U_m=\{U_\lambda:\lambda\in\Lambda_m\}$ such that each $\mathcal U_m$ refines $\mathcal W_m$ and the family $\mathcal U=\bigcup_{m=0}^{m=n}\mathcal U_m$ covers $X$. Then, repeating the arguments from the final part of the proof of Theorem 1.1, we
construct the required chain morphism $\varphi:S(X,\mathcal U;G)\to S(Y;G)$. $\Box$

\section{Fixed-point theorems for homologically $UV^n(G)$ usco maps}

In this section we prove Theorems 1.3 and 1.5.
For a set-valued map $\Phi:X\to 2^Y$ we denote by
$\mathcal O(\Phi)$ the family of the open-graph maps $\Theta:X\to 2^Y$ such that $\Phi(x)\subset\Theta(x)$ for all $x\in X$.
Next proposition is a homological version (and its proof is a small modification) of \cite[Proposition 4.2]{vg1}.

\begin{pro}
Let $X$ be a paracompact space, $Y$ be a space and let $\Phi:X\to 2^Y$ be an usco map such that for every $x\in X$ and a neighborhood $U$ of $\Phi(x)$ there exists a neighborhood $V$ of $\Phi(x)$ with $V\stackrel{H_m}{\hookrightarrow}U$. Then for every $\varphi\in\mathcal O(\Phi)$ there exists
$\Theta\in\mathcal O(\Phi)$ such that $\Theta(x)$ is open in $Y$ and $\overline{\Theta(x})\stackrel{H_m}{\hookrightarrow}\varphi(x)$ for all $x\in X$
\end{pro}

\begin{proof}
Let $\varphi\in\mathcal O(\Phi)$. Then the graph $G(\varphi)$ is open in $X\times Y$ and contains the compact set $\{x\}\times\Phi(x)$ for every $x\in X$. So, there exist neighborhoods $W_1(x)$ and $U(x)$  of $x$ and $\Phi(x)$, respectively, with $W_1(x)\times U(x)\subset G(\varphi)$. Thus,
$\Phi(x)\subset U(x)\subset\varphi(x')$ for all $x'\in W_1(x)$. Then $\overline{V(x})\stackrel{H_m}{\hookrightarrow}U(x)$ for some open neighborhood $V(x)$ of $\Phi(x)$.
Since $\Phi$ is upper semi-continuous, we can find a neighborhood $W(x)\subset W_1(x)$ such that $x'\in W(x)$ implies $\Phi(x')\subset V(x)$.  Hence, for all $x'\in W(x)$ we have
\begin{itemize}
\item[(9)] $\Phi(x')\subset\overline{V(x})\stackrel{H_m}{\hookrightarrow}U(x)\subset\varphi(x')$.
\end{itemize}
Next, let $\gamma=\{P_\alpha:\alpha\in A\}$ be a locally finite closed cover of $X$ refining the cover $\{W(x):x\in X\}$ (recall that $X$ is paracompact), and for every $\alpha$ fix $x_\alpha\in X$ such that $P_\alpha\subset W(x_\alpha)$. For every $x\in X$ the set  $A(x)=\{\alpha\in A: x\in P_\alpha\}$ is finite, and define $\Theta(x)=\bigcap\{V(x_\alpha):\alpha\in A(x)\}$. One can show that $\Theta$ is open-graph (see the proof of \cite[Proposition 4.2]{vg1}). Moreover, since $x\in\bigcap\{W(x_\alpha):\alpha\in A(x)\}$, it follows from $(9)$ that
$$\Phi(x)\subset\overline{V(x_\alpha})\stackrel{H_m}{\hookrightarrow}U(x_\alpha)\subset\varphi(x)$$ for all $\alpha\in A(x)$. This yields    $\Phi(x)\subset\Theta(x)\subset\overline{\Theta(x})\stackrel{H_m}{\hookrightarrow}\varphi(x)$.
\end{proof}

\smallskip
\textit{Proof of Theorem $1.3$.} Let $g\colon Y\to X$ be a continuous (single-valued) map. Without loss of generality, we may assume that $g(Y)=X$. We need to show that the set-valued map $\Phi_g=\Phi\circ g:Y\to 2^Y$ has a fixed-point. Suppose this is not true. So $y\not\in\Phi_g(y)$ for all $y\in Y$, or equivalently $\Phi(x)\subset Y\setminus g^{-1}(x)$ for all $x\in X$.  Consider the set-valued map $\varphi:X\to 2^Y$, $\varphi(x)=Y\setminus g^{-1}(x)$. Then  $\Phi$ is a selection for $\varphi$ and it is easily seen that $\varphi$ has an open graph.  Because $\Phi$ is homologically
$UV^{n-1}(G)$, we can apply Proposition 3.1 to find for each $m=0,1,..,n$ a set valued map $\Theta_m:X\to 2^Y$ such that
\begin{itemize}
\item $\Phi(x)\subset\Theta_0(x)$, $x\in X$;
\item $\Theta_m(x)\stackrel{H_m}{\hookrightarrow}\Theta_{m+1}(x)$ for all $x\in X$ and $m=0,..,n-1$;
\item Each $\Theta_n(x)$ is open in $Y$ and $\overline{\Theta_{n}(x})\subset\varphi(x)$, $x\in X$.
\end{itemize}
Then, according to Theorem 1.2, there exists an open cover $\mathcal U$ of $X$ and a chain morphism $\phi:S(X,\mathcal U;G)\to S(Y;G)$ such that
$\phi(S(U;G))\subset S(\Theta_n(x);G)$ for every $U\in\mathcal U$ and every $x\in U$. Consider the open cover $\mathcal U_g=g^{-1}(\mathcal U)$ of $Y$ and the chain morphism $g_\sharp:S(Y,\mathcal U_g;G)\to S(X,\mathcal U;G)$ generated by $g$. Then $\phi_g=\phi\circ g_\sharp:S(Y,\mathcal U_g;G)\to S(Y;G)$ is a chain morphism with
\begin{itemize}
\item[(10)] $\phi_g(S(g^{-1}(U);G))\subset S(\Theta_n(g(y));G)$ for all $U\in\mathcal U{~}$ and $y\in g^{-1}(U)$.
\end{itemize}
So, we can apply the homological fixed-point theorem \cite[Theorem 7]{bc} to conclude that the chain morphism $\phi_g$ has a fixed point $y_0\in Y$. This means that for any neighborhood $W\subset Y$ of $y_0$ there is a chain $c_W\in S(W;G)\cap S(Y,\mathcal U_g;G)$ such that $||\phi_g(c_W)||\cap W\neq\varnothing$.
Choose $U_0\in\mathcal U$ with $y_0\in g^{-1}(U_0)$ and let $V\subset g^{-1}(U_0)$ be a neighborhood of $y_0$. Then $c_V\in S(V;G)\subset S(g^{-1}(U_0);G)$. Thus, we have
\begin{itemize} \item[(11)] $||\phi_g(c_V)||\cap V\neq\varnothing$ and, by $(10)$, $||\phi_g(c_V)||\subset\Theta_n(g(y_0))$.
\end{itemize}
On the other hand, since $\overline{\Theta_{n}(x_0})\subset\varphi(x_0)$, where $x_0=g(y_0)$, we can choose $V$ to be so small that $V\cap\overline{\Theta_{n}(x_0})=\varnothing$. The last relation contradicts condition $(11)$. $\Box$

\smallskip
\textit{Proof of Theorem $1.5$.} The arguments from the proof of \cite[Theorem1.3]{vg1} work in our situation. For completeness, we provide a sketch.
Since $X$ can be embedded in the Hilbert cube $Q$ as a retract, we may suppose that $\Phi:Q\to 2^Q$ is a homologically $UV^\omega(G)$ usco map. Identifying $Q$ with the product
$\prod\{\mathbb I_k:k\in\mathbb N\}$, where $\mathbb I=[0,1]$, let $\pi_n:Q\to\prod\{\mathbb I_k:k\leq n\}$ be the projection onto the cube $\mathbb I^n$ and $h_n:\mathbb I^n\to Q$ be the embedding assigning to every point $x=(x_1,...,x_n)\in\mathbb I^n$ the point $h(x)$ having the same first $n$-coordinates and all other coordinates $0$. For every $n$ consider the homologically $UV^\omega(G)$ usco map $\Phi_n:\mathbb I^n\to 2^Q$ defined by $\Phi_n(x)=\Phi(h_n(x))$. Then, according to Theorem 1.3 (with $X=\mathbb I^n$, $Y=Q$, $g=\pi_n$ and $\Phi=\Phi_n$), there is a point $x^n\in Q$ with $x^n\in\Phi_n(\pi_n(x^n))$. If $x^0\in Q$ is the limit point of a convergent subsequence of $\{x^n\}_{n\geq 1}$, one can see that $x^0\in\Phi(x^0)$. $\Box$

\section{Fixed point-theorems for homological $UV^n(G)$ and $UV^\omega(G)$ decompositions}
 In this section we provide some fixed point-theorems for homological $UV^n$ or homological $UV^\omega(G)$ decompositions of compact metric $AR$'s, where $G$ is a field.
 Our results are homological analogues of for homotopical $UV^n$ and $UV^\omega$ decompositions, see \cite[Theorems 3-4]{bo} and \cite[Theorems 7.1-7.3]{vg1}.
We follow Gutev's scheme \cite{vg1} of proofs applying our Theorem 1.3, Corollary 1.4 and Theorem 1.5 instead of their homotopical versions.

By a homological $UV^n(G)$ (resp., homological $UV^\omega(G)$) decomposition of a compactum $X$ we mean an upper semi-continuous decomposition of $X$ into compact homologically $UV^n(G)$ (resp., homologically $UV^\omega(G)$) sets. The decomposition space is denoted by $X/_\sim$ and $\pi:X\to X/_\sim$ is the quotient map.

\begin{thm}
Let $X$ be a compact metric $AR$ with $\dim X\leq n$ and $X/_\sim$ be a homological $UV^{n-1}(G)$ decomposition of $X$. Then $X/_\sim$ has the fixed-point property.
\end{thm}

\begin{proof} For any map $f\colon X/_\sim\to X/_\sim$ the set-valued map $\Phi:=\pi^{-1}\circ f\circ\pi:X\to 2^X$ is usc and homologically $UV^{n-1}(G)$. Then, by Corollary 1.4, $x_0\in\Phi(x_0)$ for some $x_0\in X$. Hence, $f(\pi(x_0))=\pi(x_0)$.
\end{proof}

\begin{thm}
Let $X$ be a compact metric $AR$ and $X/_\sim$ be a homological $UV^{n-1}(G)$ decomposition of $X$ with $\dim X/_\sim\leq n$. Then $X/_\sim$ has the fixed-point property.
\end{thm}

\begin{proof} For any map $f\colon X/_\sim\to X/_\sim$ consider the set-valued map $\Phi:=\pi^{-1}\circ f:X/_\sim\to 2^X$ and apply Theorem 1.3 to find a point $x_0\in X$ with $x_0\in\Phi(\pi(x_0))$. The last equality implies $f(\pi(x_0))=\pi(x_0)$.
\end{proof}

\begin{thm}
Let $X$ be a compact metric $AR$ and $X/_\sim$ be a homological $UV^{\omega}(G)$ decomposition of $X$. Then $X/_\sim$ has the fixed-point property.
\end{thm}

\begin{proof} We repeat the proof of Theorem 4.1 using now Theorem 1.5 instead of Corollary 1.4.
\end{proof}

\textbf{Acknowledgments.} The author would like to express his gratitude to P. Semenov and S. Bogaty\v{i} for their helpful comments. The author also thanks the referee for his/her careful reading and suggesting some improvements of the paper.

\end{document}